\documentclass{amsart}

\usepackage{amsmath}
\usepackage{amsfonts}

  \usepackage{paralist}
  \usepackage{graphics} 
  \usepackage{epsfig} 

 \usepackage[colorlinks=true]{hyperref}

\hypersetup{urlcolor=blue, citecolor=red}

\newtheorem{theorem}{Theorem}[section]
\newtheorem{corollary}{Corollary}

\theoremstyle{definition}
\newtheorem{definition}[theorem]{Definition}

\newcommand{\D}[1]{{\mathbb#1}}
\newcommand{\RR}{{\D{R}}}

\newcommand{\PP}{{\D{P}}}

\newcommand{\Uc}{{\mathcal{U}}}

\title[The Hopf bifurcation with bounded noise]
      {The Hopf bifurcation with bounded noise}

\author[Ryan Botts,  Ale Jan Homburg and Todd Young]{}

\subjclass{Primary: 37H20; Secondary: 37G10, 34F20.}
 \keywords{Random dynamical system, random differential equation, stationary measure, minimal forward invariant set.}

 \email{ryanbotts@pointloma.edu}
 \email{a.j.homburg@uva.nl}
 \email{youngt@ohio.edu}

\thanks{T.Y.\ was partially supported by NIH-NIGMS grant R01GM090207.}

\begin{document}
\maketitle

\centerline{\scshape Ryan T.~Botts }
\medskip
{\footnotesize
 \centerline{Department of Mathematical, Information \& Computer Sciences}
   \centerline{Point Loma Nazarene University, 3900 Lomaland Drive}
   \centerline{ San Diego, CA 92106 USA}
} 

\medskip

\centerline{\scshape Ale Jan Homburg}
\medskip
{\footnotesize
 \centerline{KdV Institute for Mathematics, University of Amsterdam, }
   \centerline{Science Park 904, 1098 XH, Amsterdam, Netherlands, and,}
   \centerline{Department of Mathematics, VU University Amsterdam,}
   \centerline{De Boelelaan 1081, 1081 HV Amsterdam, Netherlands}
}

\medskip

\centerline{\scshape Todd R.~Young}
\medskip
{\footnotesize
 \centerline{Department of Mathematics, Ohio University, }
    \centerline{321 Morton Hall, OH 45701 Athens, USA}
}

\bigskip

 \centerline{To appear in Discrete and Continuous Dynamical Systems - A.}
 \centerline{(Communicated by Hinke Osinga)}

\begin{abstract}
We study Hopf-Andronov bifurcations in a class of random differential
equations (RDEs) with bounded noise. We observe that when an ordinary differential
equation that undergoes a Hopf bifurcation is subjected to bounded noise
then the bifurcation that occurs involves a discontinuous change in the Minimal Forward Invariant
set.
\end{abstract}


\medskip
{\em Dedicated to  John Guckenheimer;  for his many contributions to our field.}

\section{Introduction}


We will consider Hopf-Andronov bifurcations in a class of random differential
equations (RDEs)
\begin{equation}\label{rde}
\dot{x} = f_\lambda(x) + \varepsilon \xi_t
\end{equation}
as the parameter $\lambda \in \RR$ is varied.
Here $x$ will belong to the plane 
and $\xi_t$ will be a realization of some noise.
We are interested in bounded noise in which $\xi_t$ takes values in a closed disk
$\Delta\subset\RR^2$. More specifically we will consider Hopf-Andronov bifurcations with
radially symmetric noise where $\xi_t$ takes values in the unit disk. 
The RDEs without noise, for $\varepsilon=0$, unfold a
Hopf-Andronov bifurcation as $\lambda$ varies.
We denote by $\Uc$ the
collection of all possible realizations of the noise.
We assume that $f$ and $\Uc$ are sufficiently well-behaved that the
equations (\ref{rde}) uniquely defines a flow $\Phi^t(x,\xi)$
for all realizations $\xi$ of the noise.

For the general framework of random dynamical systems
we refer the reader to L.~Arnold's book \cite{Ar1} (see also \cite{CK1,Jo}). A distinctive feature
of dynamical systems with bounded noise is that they may admit more than one stationary measure.
For discrete time Markov processes these measures were studied by
Doob \cite{Do}, who showed that their supports are precisely the {\em minimal
forward invariant} (MFI) sets.
In \cite{HY_hard} the authors adapted Doob's proof
to the case of the continuous time (not necessarily Markov) processes generated by
RDEs on compact manifolds.
It was observed (\cite{HY_hard,ZH}) that these stationary measures can
experience dramatic changes, such as a change in the number of
stationary measures or a discontinuous change in one of the supports
of densities. We refer to such changes as {\em hard bifurcations}.
Given the one-to-one correspondence between stationary measures
and the MFI sets on which they are supported,
in order to study hard bifurcations, it is sufficient to study the
bifurcations of MFI sets themselves. 

We adopt from \cite{HY_2drandbif} the following assumptions on (\ref{rde}) and its flow:

\vskip .1cm
\noindent
{\bf H1}. The set $\Delta$ is a closed disk with smooth boundary. 
For each $x$ the map $\Delta \rightarrow T_x X$ given by $\xi \mapsto f(x,\xi)$
is a diffeomorphism with a strictly convex image $f(x, \Delta)$.
\vskip .1cm

\vskip .1cm
\noindent
{\bf H2}. There exist $r_0>0$ and  $t_1 >0$ such that
$$
\Phi^t_\lambda(x,\Uc) \supset B(\Phi^t_\lambda(x,0),r_0) \qquad \forall t \ge t_1.
$$
\vskip .1cm

In \cite{HY_2drandbif}, under these noise conditions the authors provided
a complete classification of bounded noise co-dimension one hard bifurcations in phase
space dimensions 1 and 2.
We call a set $F \subset X$  {\em forward invariant} if
\begin{equation}\label{invariantset}
    \Phi^t_\lambda(F,\Uc) \subset F
\end{equation}
for all $t \in \RR^+$. Denote by $\mathcal{F}$ the collection
of forward invariant sets. There is a partial ordering on
$\mathcal{F}$ by inclusion, i.e. $E \preceq F$ if $E \subset F$.
We call $E \subset \mathcal{F}$ a {\em minimal forward invariant} (MFI) set
if it is minimal with respect to the partial ordering $\preceq$.
In this context, MFI sets were shown to exist in \cite{HY_hard}.
It follows easily from the definitions that an MFI set for
(\ref{rde}) is open and connected and that the closures
of distinct MFI sets are disjoint.

A natural assumption is 
further that we are given a 
$\theta^t$-invariant, ergodic probability measure $\PP$ on $\Uc$ with the following hypothesis:

\vskip .1cm
\noindent
{\bf H3}. There exist $t_2 >0$ so that $\Phi^t_\lambda(x,\cdot)_* \PP$
is absolutely continuous w.r.t. a Riemannian measure $m$ on $X$ for all $t > t_2$ and all 
$x \in X$.
\vskip .1cm

Under this assumption, the closures of MFIs are supports of stationary densities \cite{HY_hard}.
We note that the concept of MFI set is the same
as the concept of {\em positively invariant maximal control sets} defined
in the context of control theory \cite{CK2}.
We now define bifurcation of MFI sets.
\begin{definition}
A {\em bifurcation} of MFI sets is said to occur in a parameterized family
of random differential equations if either: 
\begin{itemize}
   \item[\bf B1] The number of MFI sets changes.
   \item[\bf B2] A MFI set changes discontinuously with respect to the Hausdorff metric.
\end{itemize}
\end{definition}

In a supercritical Hopf bifurcation taking place in \eqref{rde} for $\varepsilon=0$,
a stable limit cycle appears in the bifurcation at $\lambda=0$.
For a fixed negative value of $\lambda$, the differential equations 
without noise possess
a stable equilibrium and the RDE with small noise has an MFI which is a disk around the equilibrium.
Likewise, at a fixed positive  value  of $\lambda$ for which \eqref{rde} without noise possesses
a stable limit cycle, small noise will give an annulus as MFI. 
A bifurcation of stationary measures takes place when varying $\lambda$. 
We will prove the following bifurcation scenario for small $\varepsilon >0$: the RDE \eqref{rde} undergoes
a hard bifurcation of type {\bf B2} in which a globally attracting MFI set changes
discontinuously, by suddenly developing a ``hole''.
This hard bifurcation takes place at a 
delayed parameter value $\lambda = \mathcal{O}(\varepsilon^{2/3})$.
We include a brief discussion of attracting random cycles.
Finally we demonstrate the results with numerical experiments on 
a radially symmetric Hopf bifurcation with bounded noise. 

For recent studies of Hopf bifurcations in stochastic differential equations (SDEs)
we refer to \cite{Bashetal,Arnetal,absh,Wie}.  
In such systems there is a unique stationary measure, with support equal to the entire state space.
Bifurcations of supports of stationary measures, as arising in RDEs with bounded noise, 
do not occur in the context of SDEs. 

\section{Random perturbations of a 
planar Hopf-Andronov bifurcation}

Consider a smooth family of planar random differential equations
\begin{equation}\label{e:f+euv}
(\dot{x},\dot{y}) = f_\lambda (x,y) + \varepsilon (u,v)
\end{equation}
where $\lambda\in \mathbb{R}$ is a parameter and $u,v$ are noise terms from  $\Delta = \{u^2+v^2 \le 1\}$, 
representing radially symmetric noise.
Hypothesis {\bf H1} is therefore fulfilled, we consider noise such that 
also {\bf H2} is satisfied.

We assume that without the noise terms, i.e. for $\varepsilon = 0$, the family of differential equations
unfolds  a supercritical Hopf bifurcation at $\lambda=0$ \cite{Kuz}.
The Hopf bifurcation in planar RDEs with small bounded noise is described in the following result.

\begin{theorem}\label{t:randomhopf}
Consider a family of RDEs \eqref{e:f+euv} depending on one parameter $\lambda$,
that unfolds, when $\varepsilon = 0$,  a supercritical Hopf bifurcation at $\lambda=0$.
For small $\varepsilon >0$ and $\lambda$ near 0, there is a unique MFI set $E_\lambda$. 
There is a single hard bifurcation at 
$\lambda_\mathrm{bif} = \mathcal{O}(\varepsilon^{2/3})$ as
$\varepsilon \downarrow 0$. At $\lambda =\lambda_\mathrm{bif}$ the MFI set $E_\lambda$ changes
from a set diffeomorphic to a disk for $\lambda <  \lambda_\mathrm{bif}$ to a set diffeomorphic
to an annulus  for $\lambda \ge  \lambda_\mathrm{bif}$. At $\lambda_\mathrm{bif}$ the inner radius
of this annulus is $r^* = O(\varepsilon^{1/3})$.
\end{theorem}

\begin{proof}
We first bring the system without the noise terms, 
\begin{align}\label{e:f}
(\dot{x},\dot{y}) &= f_\lambda (x,y),
\end{align}
 into normal form.
Note that a smooth coordinate transformation $h_\lambda$ and 
a multiplication by a smooth positive function $\alpha_\lambda$
(which is equivalent to 
a time reparameterization), change \eqref{e:f} into
\begin{align}\label{e:nf}
 (\dot{x},\dot{y}) &= \alpha_\lambda (x,y) (h_\lambda)_* f_\lambda (x,y)
=  \alpha_\lambda (x,y) Dh(h^{-1}(x,y))   f_\lambda(h^{-1}(x,y)) .
\end{align}
A smooth normal form transformation consists 
of a smooth coordinate transformation $h_\lambda$, a time reparameterization, 
and an additional reparametrization of the parameter $\lambda$.
There exists a smooth normal form transformation that changes \eqref{e:f} into
\begin{align}
 \left( \begin{array}{c} \dot{x} \\ \dot{y} \end{array}\right) &= \left( \begin{array}{cc} \lambda & -1 \\ 1 & \lambda \end{array} \right) \left( \begin{array}{c} {x} \\ {y} \end{array}\right)- r^2 \left( \begin{array}{c} {x} \\ {y} \end{array}\right) + \mathcal{O}(r^4),
\end{align}
see \cite{Kuz} (here $r = \sqrt{x^2+y^2}$).
Applying this normal form transformation
to \eqref{e:f+euv} yields
\begin{align}
 \left( \begin{array}{c} \dot{x} \\ \dot{y} \end{array}\right) &= \left( \begin{array}{cc} \lambda & -1 \\ 1 & \lambda \end{array} \right)\left( \begin{array}{c} {x} \\ {y} \end{array}\right) - r^2 \left( \begin{array}{c} {x} \\ {y} \end{array}\right)+ \mathcal{O}(r^4) + \varepsilon A(x,y,\lambda) \left(\begin{array}{c} u \\ v \end{array}\right)
\end{align}
where $A(x,y,\lambda) = A_0(\lambda) + \mathcal{O} (r)$ is a two by two matrix that depends smoothly on its arguments.
Note that expression \eqref{e:nf} implies that the noise terms $(u,v)$ are the same as in \eqref{e:f+euv}; 
they only get multiplied by a matrix $A(x,y,\lambda)$.
Changing to polar coordinates we find
\begin{align*}
 \dot{r} &= \lambda r - r^3 + \mathcal{O}(r^4) + \varepsilon \left\langle A \left(\begin{array}{c} u \\ v \end{array}\right) , \left(\begin{array}{c}\cos \theta \\ \sin \theta\end{array}\right) \right\rangle,
\\
\dot{\theta} &= 1 + \frac{\varepsilon}{r} \left\langle A \left(\begin{array}{c} u \\ v\end{array}\right) , \left(\begin{array}{c} -\sin \theta \\ \cos \theta\end{array}\right) \right\rangle.
\end{align*}
We perform a rescaling in $\varepsilon$ by putting $r = \varepsilon^{1/3} \bar{r}$ and
$\lambda = \varepsilon^{2/3} \bar{\lambda}$. This brings
\begin{align}
\label{e:thetabarr}
\begin{split}
 \dot{\bar{r}} &= \varepsilon^{2/3} \left[ \bar{\lambda} \bar{r} - \bar{r}^3 + \varepsilon^{1/3}\mathcal{O}(\bar{r}^4) +  \left\langle A \left(\begin{array}{c} u \\ v\end{array}\right)  , \left(\begin{array}{c}\cos \theta \\\sin \theta\end{array}\right) \right\rangle \right],
\\
\dot{\theta} &= \varepsilon^{2/3} \left[ \varepsilon^{-2/3} + \frac{1}{\bar{r}}  \left\langle A \left(\begin{array}{c} u \\ v\end{array}\right) , \left(\begin{array}{c}-\sin \theta \\ \cos \theta\end{array}\right) \right\rangle \right],
\end{split}
\end{align}
were
$A = A_0(\lambda) + \varepsilon^{1/3}\mathcal{O}(\bar{r})$.
Multiplying by a factor $\varepsilon^{-2/3}$ and taking 
the limit $\varepsilon \downarrow 0$ gives
\begin{align*}
 \dot{\bar{r}} &= \bar{\lambda}\bar{r} - \bar{r}^3 + \langle A (u,v) , (\cos \theta,\sin \theta) \rangle
\end{align*}
for the radial component.

Noting that $(u,v)$ lies in a unit disk, $A (u,v)$ lies in an ellipse. If we let $a$ and $b$ be the
major and minor axes of this ellipse, then by a rotation of coordinates we may transform the
last equation into
\begin{align}\label{rbarsimple}
 \dot{\bar{r}} &= \bar{\lambda}\bar{r} - \bar{r}^3 + a u \cos \theta + b v \sin \theta.
\end{align}
In \cite{HY_2drandbif} the authors showed that the boundary of an MFI set consists
of solutions of the extremal vector fields defined by the bounded noise differential equations.
Observe that for $\varepsilon = 0$, \eqref{e:thetabarr} reads $\dot{\theta} = 1$ and $\dot{\bar{r}} = 0$ 
and its right hand varies continuously with $\varepsilon$.
Hence for $\varepsilon$ small, the $\bar{r}$ components of the extremal vector fields 
are approximately the extremal values of the $\bar{r}$ component.
A simple calculation shows that the extremal values of the last two terms in (\ref{rbarsimple}) are given by:
$$
                 \pm \sqrt{a^2 \cos^2 \theta + b^2 \sin^2 \theta}.
$$
Thus the boundaries of all solutions of (\ref{rbarsimple}) are given by the solution $r_\pm$ of the equations:
 \begin{equation}\label{rpm}
 \dot{r}_\pm = \bar{\lambda} r_\pm - r_{\pm}^3  \pm \sqrt{a^2 \cos^2 \theta + b^2 \sin^2 \theta}. 
\end{equation}
Now note that $\dot{\theta}$ is of order $\varepsilon^{-2/3}$.  
This allows us to average equations \eqref{rpm} 
to obtain the averaged equations:
 \begin{equation}
 \dot{\bar{r}}_\pm = \bar{\lambda} \bar{r}_\pm - \bar{r}_{\pm}^3  
         \pm \int_0^{2 \pi} \sqrt{a^2 \cos^2 \theta + b^2 \sin^2 \theta} \, d\theta. 
\end{equation}
Note that the integral may be transformed into:
$$
     4 a \int_0^{\pi/2} \sqrt{1 - \left(1 - \frac{a^2}{b^2}\right) \sin^2 \theta } \, d\theta
         = 4 a E\left(\sqrt{1 - a^2/b^2}\right),
$$
where $E(k) \equiv \int_0^{\pi/2} \sqrt{1 - k^2 \sin^2 \theta } \, d\theta$ is the complete
elliptic integral of the second kind. Thus the averaged equations for $r_\pm$ are simply:
\begin{equation}\label{rpmavg}
 \dot{\bar{r}}_\pm = \bar{\lambda} \bar{r}_\pm - \bar{r}_{\pm}^3  
         \pm  c, 
\end{equation}
where $c = 4 a E(\sqrt{1 - a^2/b^2})$.

The analysis of the MFI  bifurcation is now straightforward. 
First consider the equation for $\bar{r}_+$. For all 
parameters, this has a hyperbolic attracting fixed  point at the largest real solution 
of $r^3 - \lambda r - c = 0$. According to Theorem~4.1.1 in \cite{GH} the original equation
has a hyperbolic attracting periodic orbit that passes close to the hyperbolic fixed point
on a Poincar\'{e} section. This periodic orbit $\gamma^+$ encloses an MFI set. 
Further it is can be shown
that any orbit beginning outside of $\gamma^+$ 
will converge to the MFI set.

Next consider the equation for $\bar{r}_-$.
A saddle-node bifurcation occurs in the averaged equation when 
\begin{equation}\label{e:l*}
      \lambda_\mathrm{bif} = \frac{3 c^{2/3}}{4^{1/3}}
\end{equation}
at 
\begin{equation}\label{e:r*}
       r^* =  \left(\frac{c}{2}\right)^{\frac{1}{3}}
\end{equation}
in which a stable and an unstable fixed point are formed. 
It follows from Theorem~4.3.1 in \cite{GH} that the Poincar\'{e} map for original equation
also undergoes a saddle-node bifurcation near the one for the averaged equation.
Thus a stable periodic orbit $\gamma^-$ for $r_-$ is created. 
The stable orbits $\gamma^+$ for $r_+$ and $\gamma^-$ for $r_-$ 
enclose an annular MFI set.

We summarize the above calculations. 
For $\lambda < \lambda_\mathrm{bif}$
the disk region inside $\gamma^+$ is a minimal forward invariant set. 
For $\lambda \ge \lambda_\mathrm{bif}$
the MFI set is the annular region bounded by $\gamma^-$ and $\gamma^+$.
Thus the MFI set changes discontinuously at
$\lambda_\mathrm{bif}$ and so a hard bifurcation of type {\bf B2}  occurs.
\end{proof}


\section{Random cycles}

The deterministic Hopf bifurcation involves the creation of a limit cycle.
In this section we comment on the occurrence of attracting random cycles.

Random cycles are defined in analogy with random fixed points \cite{Ar1};
for its definition we need to consider the noise realizations $\xi$ 
and the flow $\Phi^t_\lambda (x,\xi)$ for two sided time 
$t \in \mathbb{R}$.
We henceforth consider the skew product flow
\[ 
(x,\xi) \mapsto  \left( \Phi^t_\lambda (x,\xi_t) ,  \theta^t \xi \right)
\]
with
\[
 \theta^t \xi_s = \xi_{t+s}
\]
for $t,s \in \mathbb{R}$.
We formulate a result with a relaxed notion of attracting random cycle, as we allow for 
time reparameterizations. The notion implies though that a Poincar\'e return map on a section 
$\{ \theta = 0\}$ has an attracting random fixed point.

Recall that a random fixed point is a map $R : \mathcal{U} \to \mathbb{R}^2$
that is flow invariant,
\begin{align*}
  \Phi_\lambda^t (R(\xi),\xi) &= R(\theta^t \xi)
\end{align*}
for $\mathbb{P}$ almost all $\xi$.
A random cycle is defined as a continuous map $S : \mathcal{U} \times \mathbb{S}^1 \to \mathbb{R}^2$
that gives an embedding of a circle for $\mathbb{P}$ almost all $\xi \in \mathcal{U}$ 
and is flow invariant in the sense
\begin{align*} 
(S(\theta^t \xi , \mathbb{S}^1) , t) &\subset 
\bigcup_{t \in \mathbb{R}} (\Phi^t_\lambda (S (\xi, \mathbb{S}^1) ,\xi), t) 
\end{align*}
Different regularities of the embeddings $\mathbb{S}^1 \mapsto S (\xi, \mathbb{S}^1)$ may be 
considered. 


The following result establishes the occurrence of attracting random cycles
following the hard bifurcation, for small noise amplitudes. We confine ourselves with a statement
on continuous random cycles, thus avoiding for instance constructions of invariant cone fields to prove
Lipschitz continuity. 

\begin{theorem}\label{t:rcycle}
For values of $(\lambda,\varepsilon)$ with $\lambda > \lambda_\mathrm{bif}$ and $\varepsilon$ small, 
the MFI $E_\lambda$ is diffeomorphic to an annulus and the
 flow $\Phi^t_\lambda$ admits a random cycle
$S :\mathcal{U} \times \mathbb{S}^1 \to \mathbb{R}^2$ inside $E_\lambda$.


The random cycle is attracting in the sense that there
is a neighborhood $U_\lambda$ of the MFI $E_\lambda$,
so that for all $x\in U_\lambda$, the distance between $(\Phi^t_\lambda (x,\xi),t)$ 
and $\bigcup_{t \in \mathbb{R}} (S(\theta^t \xi , \mathbb{S}^1) , t)$
goes to zero as $t\to\infty$.
\end{theorem}

\begin{proof}
Recall from the proof of Theorem~\ref{t:randomhopf} 
the blow-up differential equations in polar coordinates:
\begin{align}\label{e:t}
\begin{split} 
 \dot{\bar{r}} &=  \bar{\lambda} \bar{r} - \bar{r}^3 + \varepsilon^{1/3}\mathcal{O}(\bar{r}^4) +  \left\langle A \left(\begin{array}{c} u \\ v\end{array}\right)  , \left(\begin{array}{c}\cos \theta \\\sin \theta\end{array}\right) \right\rangle,
\\
\dot{\theta} &=  \varepsilon^{-2/3} + \frac{1}{\bar{r}}  \left\langle A \left(\begin{array}{c} u \\ v\end{array}\right) , \left(\begin{array}{c}-\sin \theta \\ \cos \theta\end{array}\right) \right\rangle,
\\
\dot{t} &= 1. 
\end{split}
\end{align}
where
$u,v$ are noise realizations.

Consider a reparameterization of time,
$ \tau = g(\bar{r},\theta,u,v) t$ with
\begin{align*}
g(\bar{r},\theta,u,v) &= 
1 + \varepsilon^{2/3} \frac{1}{\bar{r}}  \left\langle A \left(\begin{array}{c} u \\ v\end{array}\right) , \left(\begin{array}{c}-\sin \theta \\ \cos \theta\end{array}\right) \right\rangle.
\end{align*}
Note that $g$ is close to $1$, in particular positive, for small values of $\varepsilon$.
The reparameterization yields differential equations, for functions $\bar{r},\theta, t$ of $\tau$, 
\begin{align}\label{e:tau}
\begin{split}
{\bar{r}}' &=  \frac{1}{g(\bar{r},\theta,u,v)}  
\left( 
\bar{\lambda} \bar{r} - \bar{r}^3 + 
\varepsilon^{1/3}\mathcal{O}(\bar{r}) + 
 \left\langle A \left(\begin{array}{c} u \\ v\end{array}\right)  , 
\left(\begin{array}{c}\cos \theta \\\sin \theta\end{array}\right) \right\rangle 
\right),
\\
{\theta}' &=  \varepsilon^{-2/3},
\\
{t}' &= \frac{1}{g(\bar{r},\theta,u,v)}.
\end{split}
\end{align}

Note that the time reparameterization does not preserve the flow of individual cycles $S(\xi,\mathbb{S}^1)$.
However, for a fixed noise realization $\xi$, 
the graph $\bigcup_{t\in \mathbb{R}} (S(\theta^t \xi , \mathbb{S}^1),t)$
is an invariant manifold for both \eqref{e:t} and \eqref{e:tau}.
As the time parameterization has changed only by a factor $\frac{1}{g(\bar{r},\theta,u,v)}$
that is bounded and bounded away from zero,   
to prove the existence of an attracting random cycle it suffices to do this for \eqref{e:tau}.


 The differential equations \eqref{e:tau} define a skew product flow
 \begin{align*}
  (r,\theta,\xi) &\mapsto (\Psi^\tau_{\lambda} (r,\theta,\xi), \theta^\tau \xi).
 \end{align*}

Define a graph transform $\Gamma^\tau_{\lambda}$ on embedded circles
written as graphs $\mathcal{U} \times [0,2\pi] \mapsto [R_-,R_+]$ 
for suitable $0< R_- < R_+$.
That is, $\Gamma^\tau_{\lambda}$ is determined through the property
\[
\mathrm{graph}\; \Gamma^\tau_{\lambda} T (\xi,\cdot)   = 
\Psi^\tau_{\lambda}( \mathrm{graph}\; T (\theta^{-\tau} \xi, \cdot), \theta^{-\tau}\xi) 
\]
Pick $R_-,R_+$ such that, in the limit $\varepsilon = 0$, $\{ R_- < r < R_+\}$ is invariant.
This is possible for $\bar{\lambda} > \lambda_\mathrm{bif}$; note $R_- < r_- < r_+ < R_+$. 
Note that $\Gamma^\tau_{\lambda,\varepsilon}$ maps $C^0([0,2\pi], [R_-,R_+])$
into itself.
%
We obtain $S(\xi , \cdot)$ as the limit
\begin{align}\label{e:backward}
 S(\xi , \cdot) = \lim_{\tau\to\infty} \Gamma^\tau_{\lambda} T (\xi,\cdot)
\end{align}
computed in the supnorm on $C^0([0,2\pi], [R_-,R_+])$.


Consider the translation $s = \bar{r} - R_-$. Then $s$ satisfies the differential equation
\begin{align*}
 s' &=  (\bar{\lambda} - 3 R_-^2) s - 3 R_- s^2 - s^3 + \bar{\lambda} R_- - R_-^3 + a,
\end{align*}
where we abbreviated 
$
a = \left\langle A \left(\begin{array}{c} u \\ v\end{array}\right)  , 
\left(\begin{array}{c}\cos \theta \\\sin \theta\end{array}\right) \right\rangle.
$
We consider the differential equation for $s$ on $[0,R_+-R_-]$, 
which is invariant by the choices of $R_-, R_+$.
Note also that 
the coefficient  $\bar{\lambda} - 3 R_-^2$ is negative.
Suppose $s_1 > s_2$ are two solutions and consider the difference $u = s_1 - s_2$.
Then $u$ satisfies
\begin{align*}
 u' &= (\bar{\lambda} - 3 R_-^2 - 6 R_- s_2  - 3 s_2^2 ) u - 
 (3 R_-  +  3 s_2)   u^2  - u^3.
\end{align*}
All coefficients here are negative, implying that $u(\tau) \to 0$ as $\tau \to \infty$.
This computation
demonstrates the convergence in \eqref{e:backward}.
Likewise for small
values of $\varepsilon$.
\end{proof}

Write $\bar{r} \mapsto \Pi_\lambda (\bar{r} , \xi)$ for the first return map on $\{ \theta = 0\}$
of $\Phi_t (\bar{r} , 0 , \xi)$.
As a corollary of the above theorem we obtain an attracting random fixed point for $\Pi_\lambda$.

\begin{corollary}
For $\lambda$ and $\varepsilon$ as in Theorem~\ref{t:rcycle},
the first return map $\Pi_\lambda$ admits a random fixed point
$R (\xi)$ 
inside $E_\lambda \cap \{ \theta = 0\}$.

The random fixed point is attracting in the sense that there
is a neighborhood $U_\lambda$ of $E_\lambda$,
so that for all $x\in U_\lambda \cap \{ \theta = 0\}$, 
$| \Pi^k (x,\xi) - \Pi^k  (R(\xi))| \to 0$ as $k\to\infty$.
\end{corollary}

\begin{proof}
Just note that $\Pi_\lambda = \Psi_\lambda^{2\pi\varepsilon^{2/3}}$, defined 
in the proof of Theorem~\ref{t:rcycle}.
\end{proof}

A better notion of random attracting cycle would be without time reparameterizations;
\begin{align*}
 \Phi^t_\lambda (S(\xi,\mathbb{S}^1),\xi) &= S(\theta^t \xi,\mathbb{S}^1). 
\end{align*}
This would allow a discussion of the dynamics on the random cycle, such as the possibility to 
find an attracting random fixed point on it, compare \cite{absh,Arnetal}.


\section{Simulations of bounded noise, radially symmetric Hopf.}

\subsection{Radially symmetric Hopf with bounded noise}

In this section we reproduce invariant measures for the radially symmetric
system
\begin{align}\label{sym_hopf_cart}
\begin{split}
  \dot{x} &= \lambda x - y - x(x^2+y^2) + \varepsilon u, \\
  \dot{y} &= x + \lambda y - y(x^2+y^2) + \varepsilon v, 
\end{split}
\end{align}
where $\lambda \in \RR$ is a parameter and $u$ and $v$ are noise terms satisfying: $u^2+v^2 \le 1$,
representing radially symmetric noise. In the simulations below we generate the noise by 
stochastic differential equations with reflective boundary conditions.

For $\varepsilon = 0$ the differential equations \eqref{sym_hopf_cart} undergo a 
supercritical Hopf-Andronov bifurcation
at $\lambda = 0$. For $\lambda <0$ the origin is a stable
global attractor. For $\lambda > 0$ the origin is unstable, but there is a
circular globally attracting periodic orbit at $r = \sqrt{\lambda}$. 
Changing to polar coordinates we find:
\begin{align}\label{sym_hopf_polar}
\begin{split}
  \dot{r} &= \lambda r - r^3 + \varepsilon \alpha, \\
  \dot{\theta} &= 1  + \frac{\varepsilon}{r} \beta, 
\end{split}
\end{align}
where
$$
  \alpha = (u \cos \theta + v \sin \theta)
             = \langle u,v \rangle \cdot \langle \cos \theta, \sin \theta \rangle
$$
and
$$
  \beta = (-u \sin \theta + v \cos \theta)
             = \langle v,-u \rangle \cdot \langle \cos \theta, \sin \theta \rangle.
$$
Figure~\ref{upperlower} indicates boundaries of radial components of the MFI set for
\eqref{sym_hopf_polar}. Following the proof of Theorem~\ref{t:randomhopf}
the boundaries of the MFI set are given by zeros of the upper and lower 
radial differential equations 
\begin{equation}
\dot{r}_\pm = \lambda r_\pm - r^3_\pm \pm \varepsilon.
\end{equation}

From (\ref{e:l*}) and (\ref{e:r*}),
$$
      \lambda_\mathrm{bif} = \frac{3 \varepsilon^{2/3}}{4^{1/3}},   \qquad
       r^* =  \left(\frac{\varepsilon}{2}\right)^{\frac{1}{3}}.
$$

Write $\rho^+$ for the positive zero of the upper differential equation for $r_+$.
For $\lambda \ge \lambda_\mathrm{bif}$,
let $\rho_-$ be the largest positive zero of the lower differential equation for $r_-$. 
For $\lambda<\lambda_\mathrm{bif}$ the MFI set is a disk of radius $\rho_+$.
For $\lambda\ge \lambda_\mathrm{bif}$ the MFI set is an annulus bounded by circles with radii
$\rho_-,\rho_+$. See Figure~\ref{upperlower}.

\begin{figure}[th]
\centerline{
\hbox{{\psfig{figure=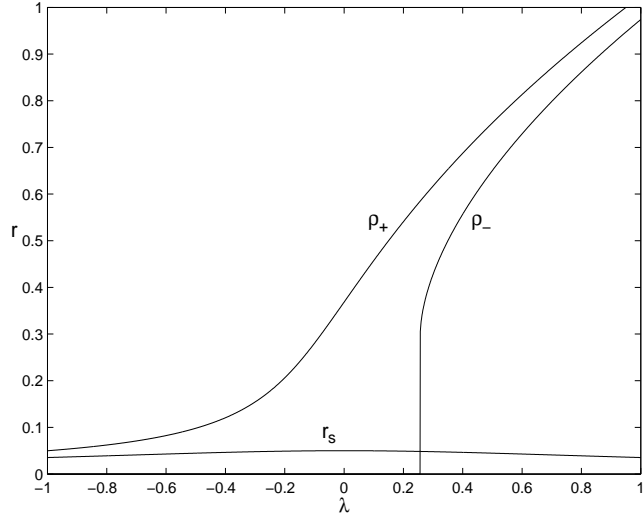,height=3in,width=4in}}}}
\caption{Boundaries of the radial components of the minimal invariant set
in the random symmetric Hopf bifurcation \eqref{sym_hopf_cart}. 
Here $\varepsilon = 0.05$.
The hard bifurcation takes place at  
$\lambda_\mathrm{bif} \approx 0.2565$ and the inner boundary has initial radius $r^* \approx 0.2924$.
The MFI set changes from a disk with radius $\rho_+$
for $\lambda<\lambda_\mathrm{bif}$ to an annulus with radii $\rho_-<\rho_+$ for 
$\lambda \ge \lambda_\mathrm{bif}$.
The graph labeled $r_s$ indicates
the boundary of the set of fixed points which are stable for $\lambda \le 0$ and unstable
for $\lambda >0$. }
\label{upperlower}
\end{figure}

In addition to the MFI set, another important dynamical feature of these equations is the set of
equilibrium points. For $\lambda \le 0$ these equilibrium points are stable, but for $\lambda >0$ they
are unstable. These are the points that are fixed points of (\ref{sym_hopf_cart}) when $(u,v)$ is a fixed value.
One can find that this set is a disk centered at the origin with radius $r_s$ which is a solution of:
$$
(1 - 2 \lambda) r^6 + (1 + \lambda^2) r^2 - \varepsilon^2 = 0.
$$
The radius $r_s$ as function of $\lambda$ is plotted in Figure~\ref{upperlower}.

\subsection{ Simulations}

In this section we simulate the bounded noise bifurcation for the sake of demonstration.
We note that associated with MFI sets are invariant densities whose supports are the MFI sets
\cite{HY_hard}. In the symmetric Hopf bifurcation as in the previous subsection, we
approximate these invariant densities for different values of $\lambda$ as $\lambda$ moves
through the random bifurcation.

The generation of bounded noise is
somewhat arbitrary without specific knowledge about the noise involved in a
particular setting. In this work we generate the noise terms $u$ and $v$ via
the stochastic system:
\begin{align}\label{sde1}
\begin{split}
  du &= \sigma \, dW_1, \\
  dv &= \sigma \, dW_2, 
\end{split}
\end{align}
where $dW_1$ and $dW_2$ are independent (of each other) normalized white noise processes  and (\ref{sde1}) are
interpreted in the usual way as It\={o} integral equations. 
 In order to assure boundedness and radial symmetry, we restrict $(u,v)$ to
the unit disk and impose reflective boundary conditions. Other methods of
generating and bounding the noise did not produce significantly difference
in the results.

The parameter $\sigma$ can be interpreted as
the rapidity of the noise. If $\sigma$ is small, then $u$ and $v$ change slowly
and as $\sigma$ increases, they change more quickly.
It turns out in our simulations  that the value of $\sigma$
has a strong influence on the characteristics of the invariant density
of (\ref{sym_hopf_cart}).

\begin{figure}[t!]

\centerline{
\hbox{{\psfig{figure=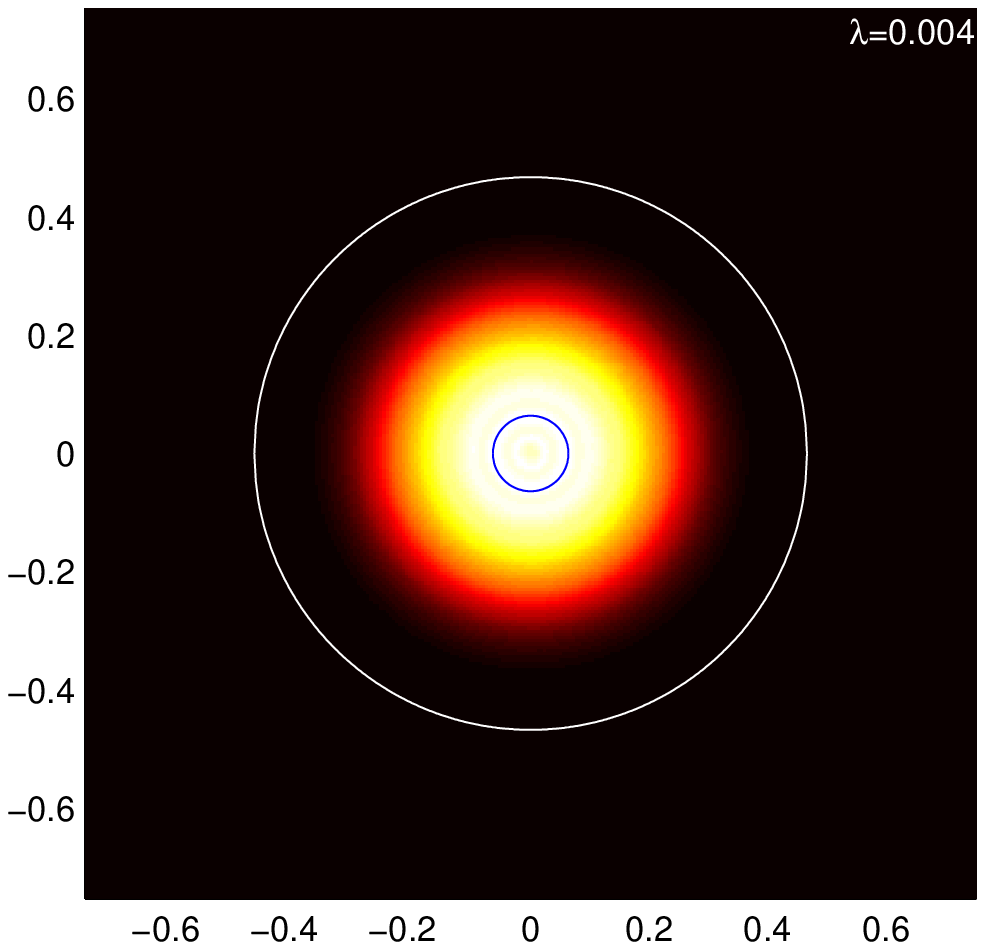,height=1.5in,width=2in}}}
\hspace*{-.3in}\hbox{{\psfig{figure=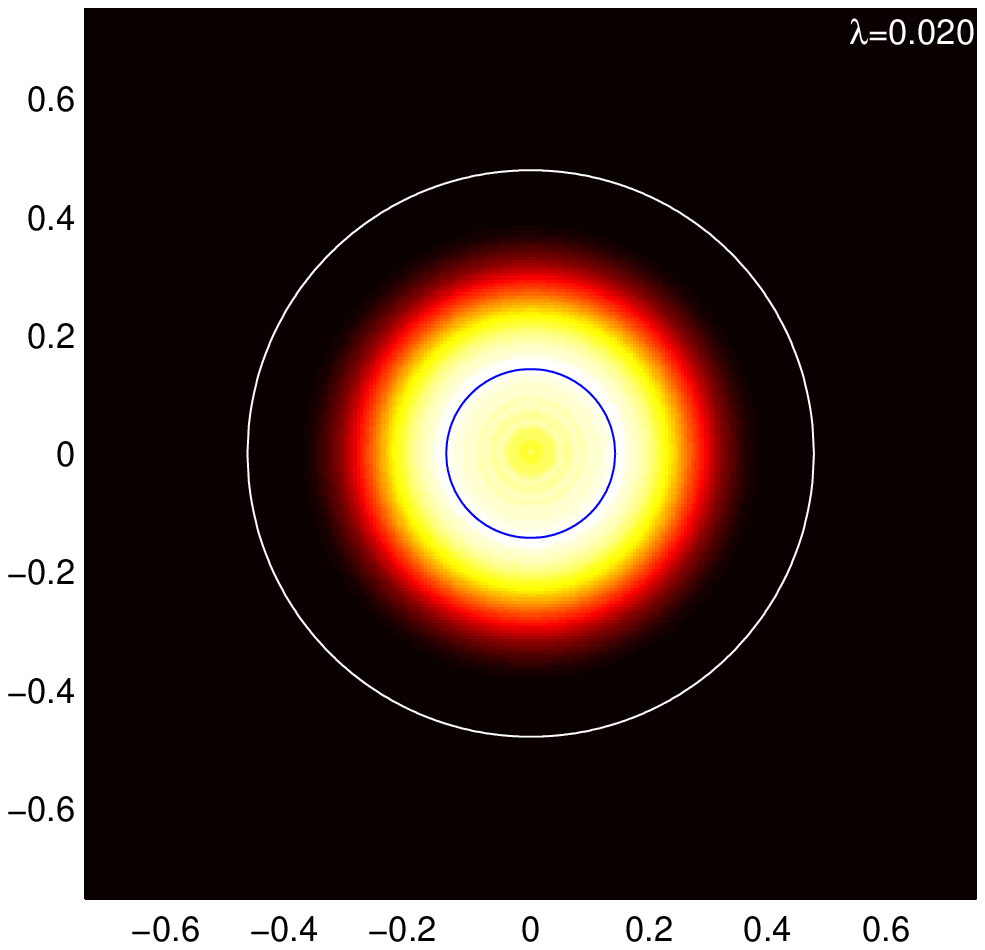,height=1.5in,width=2in}}}
\hspace*{-.3in}\hbox{{\psfig{figure=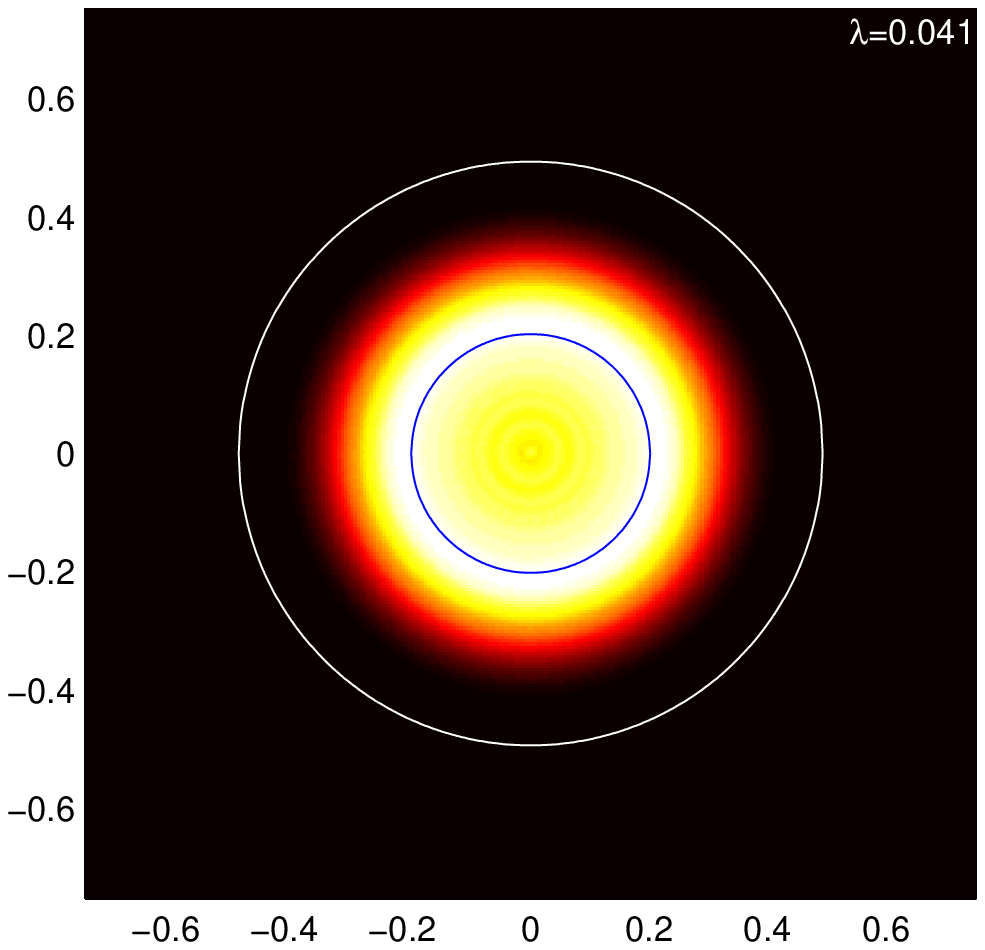,height=1.5in,width=2in}}}
}
\centerline{
\hbox{{\psfig{figure=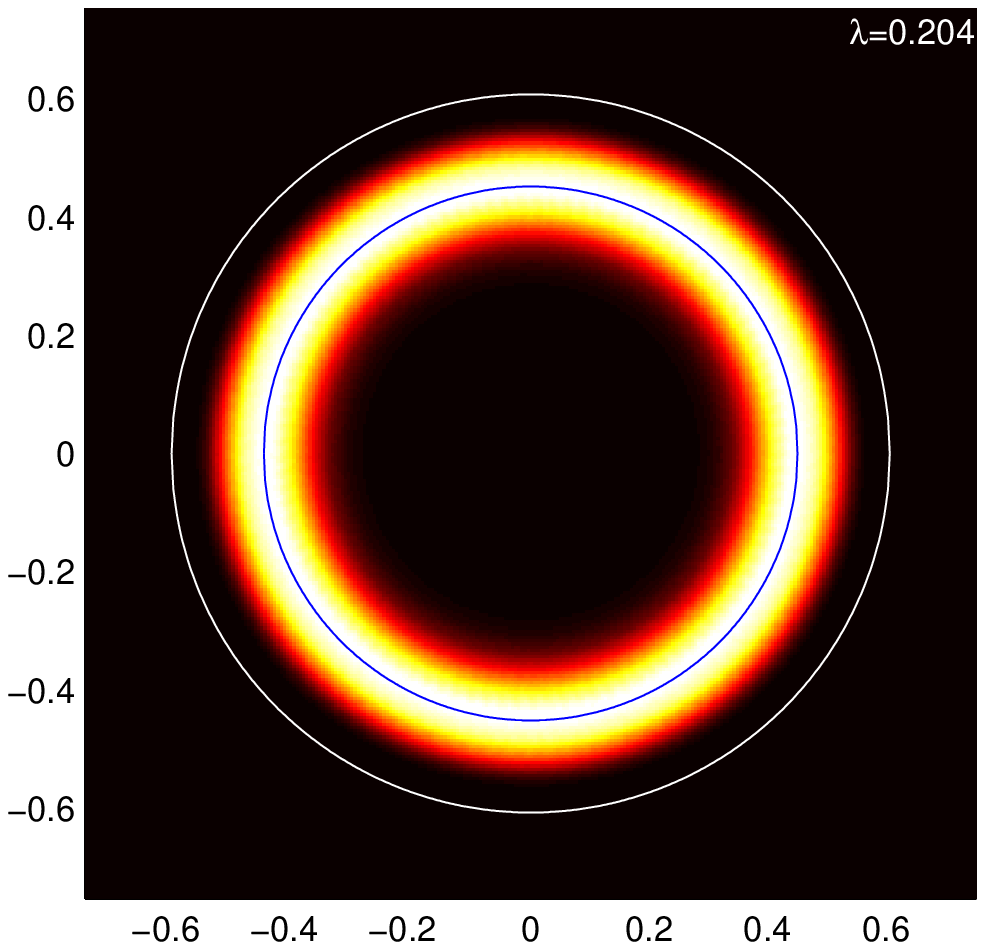,height=1.5in,width=2in}}}
\hspace*{-.3in}\hbox{{\psfig{figure=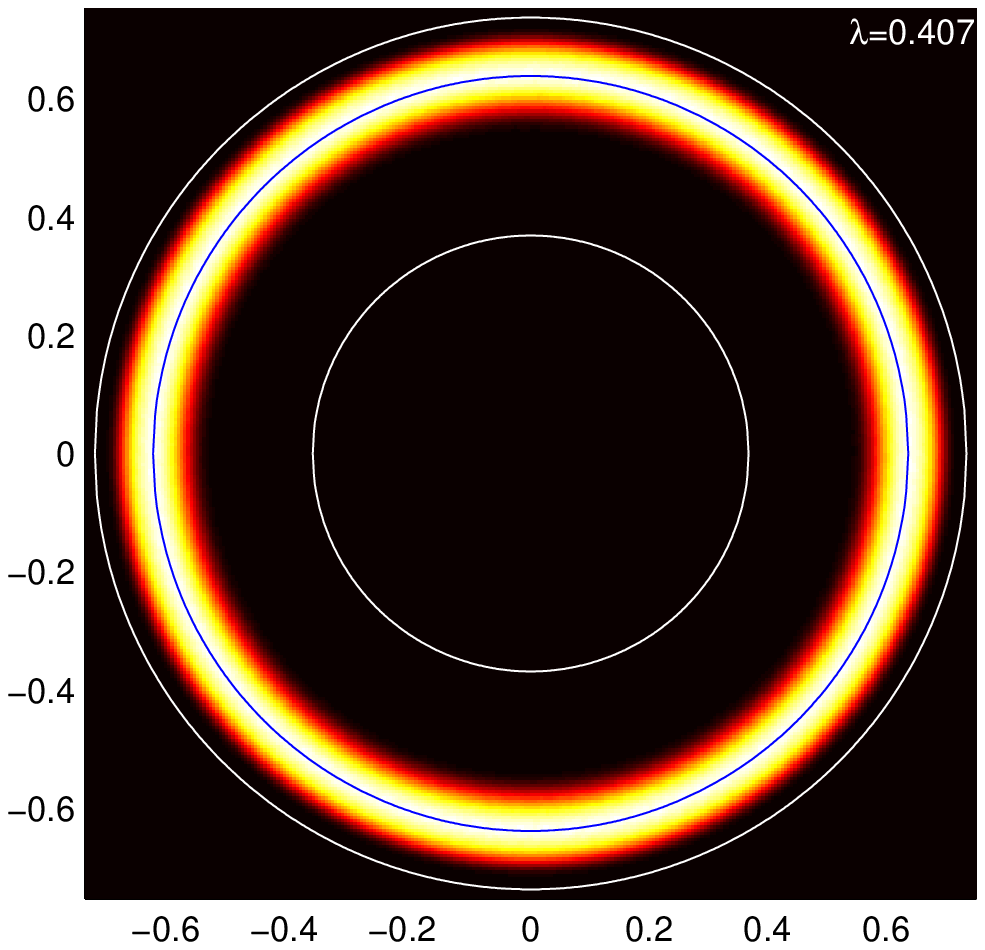,height=1.5in,width=2in}}}
\hspace*{-.3in}\hbox{{\psfig{figure=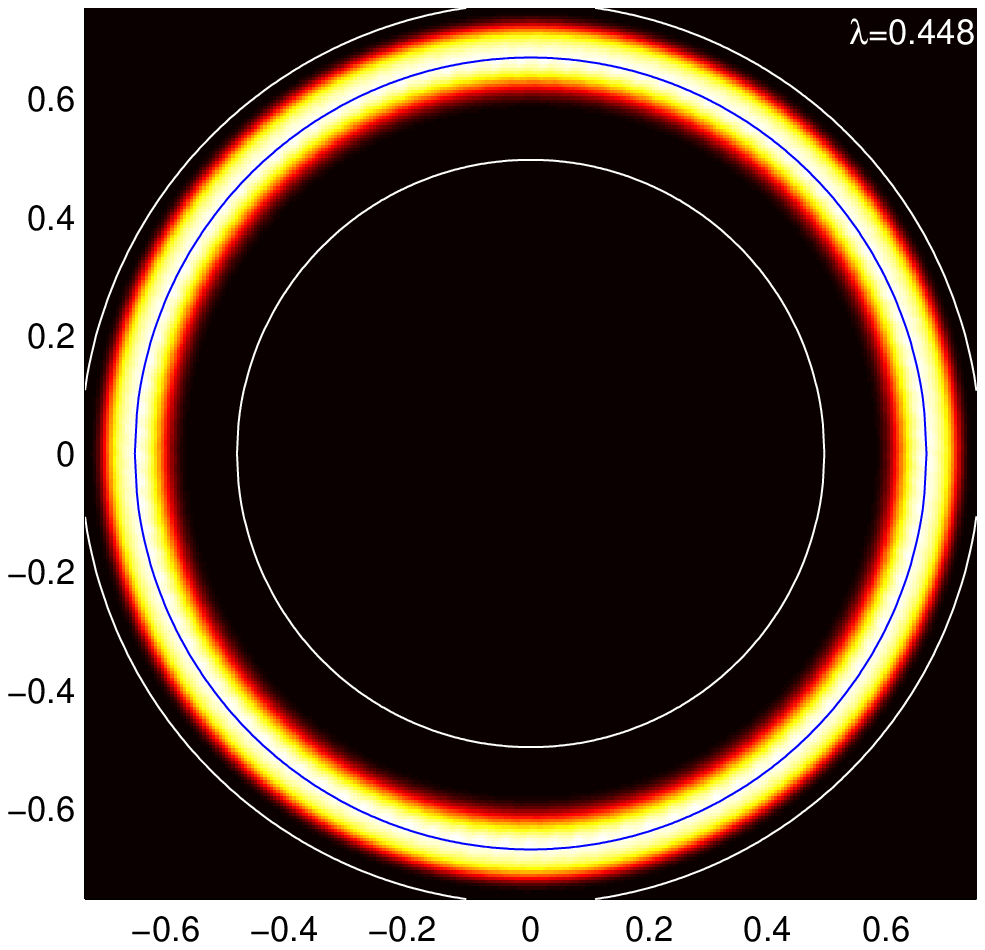,height=1.5in,width=2in}}}
}
\caption{Invariant densities for fast noise: $\sigma =1$ and $\varepsilon = 0.1$. 
}
\label{fig:fast}
\medskip

\centerline{
\hbox{{\psfig{figure=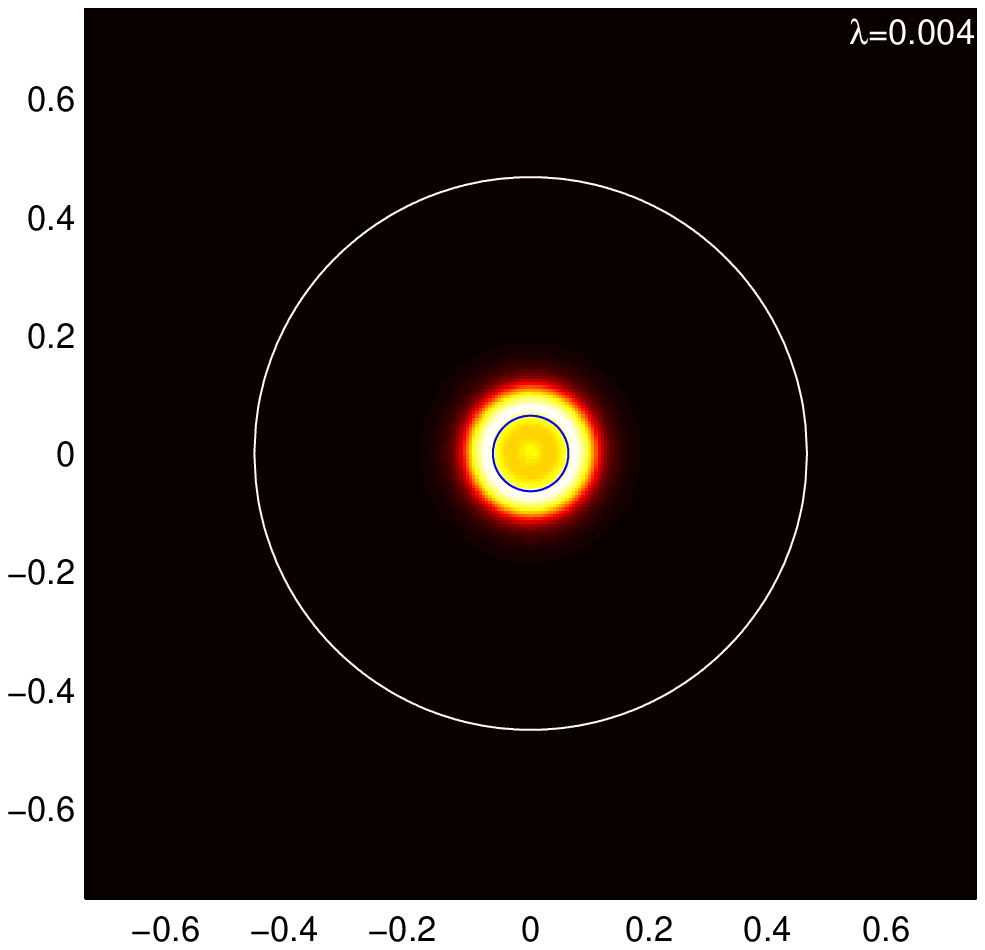,height=1.5in,width=2in}}}
\hspace*{-.3in}\hbox{{\psfig{figure=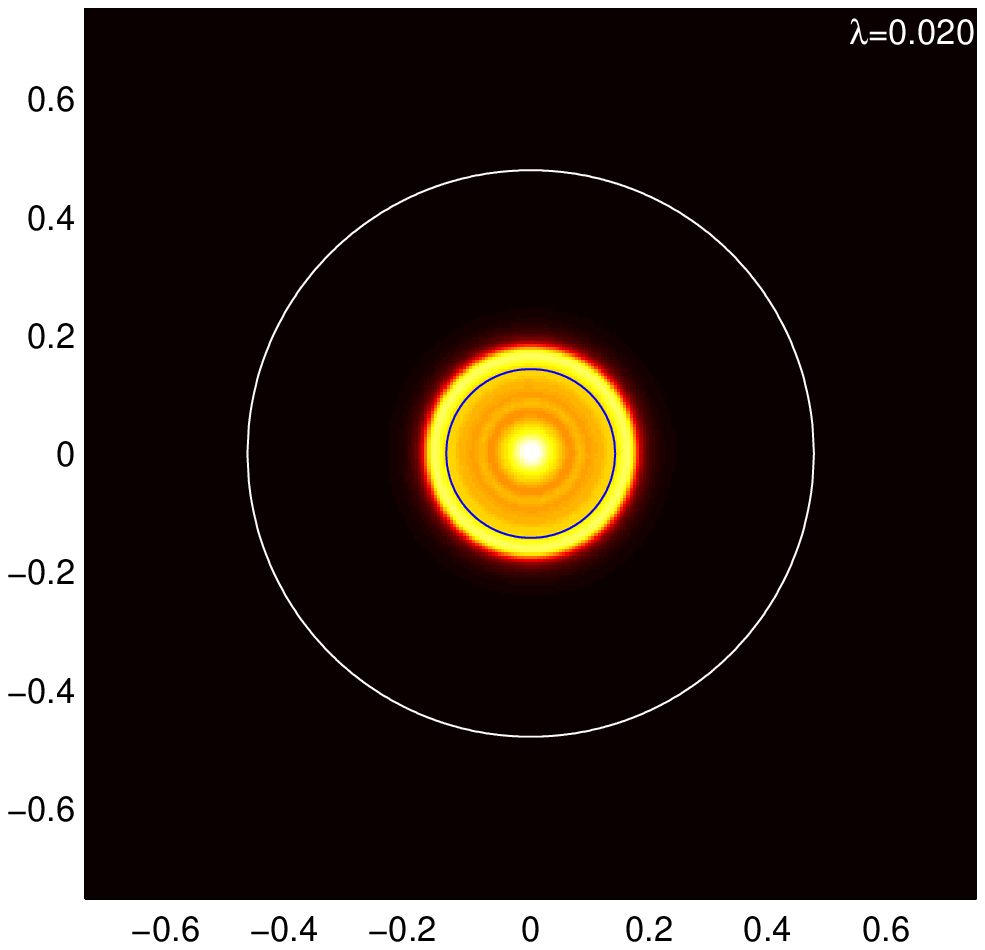,height=1.5in,width=2in}}}
\hspace*{-.3in}\hbox{{\psfig{figure=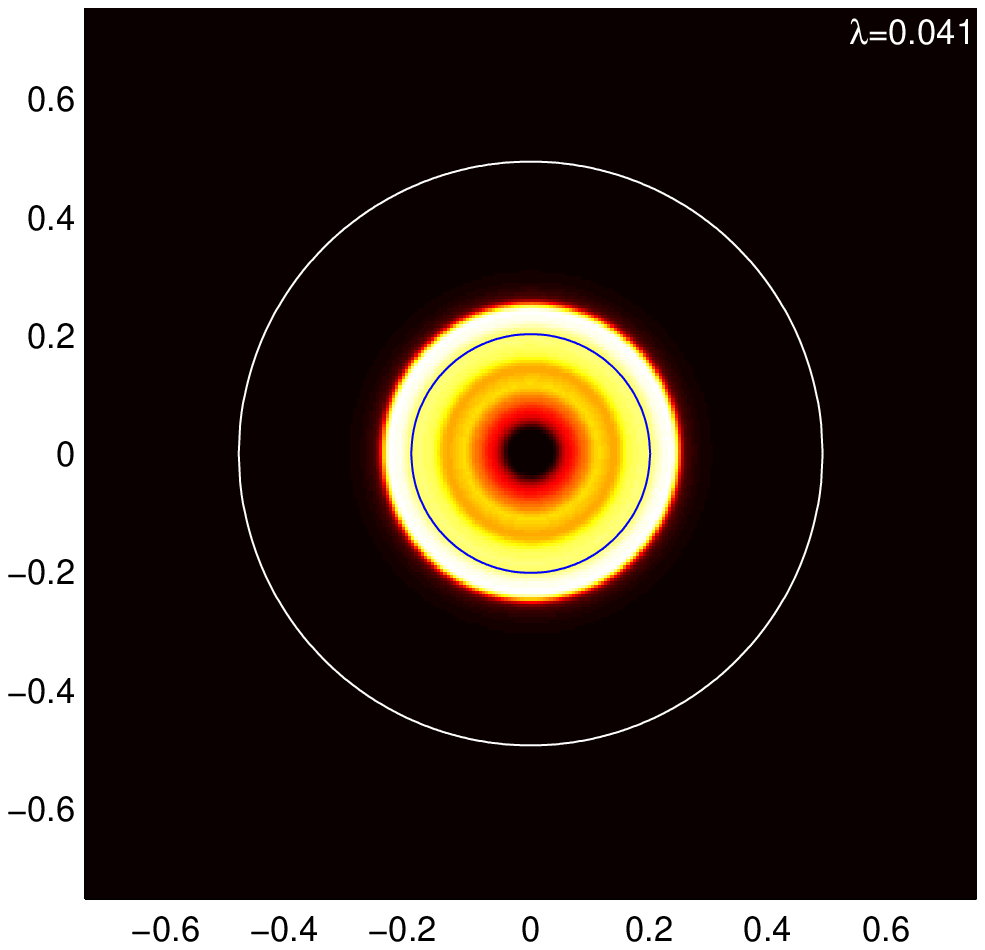,height=1.5in,width=2in}}}
}
\centerline{
\hbox{{\psfig{figure=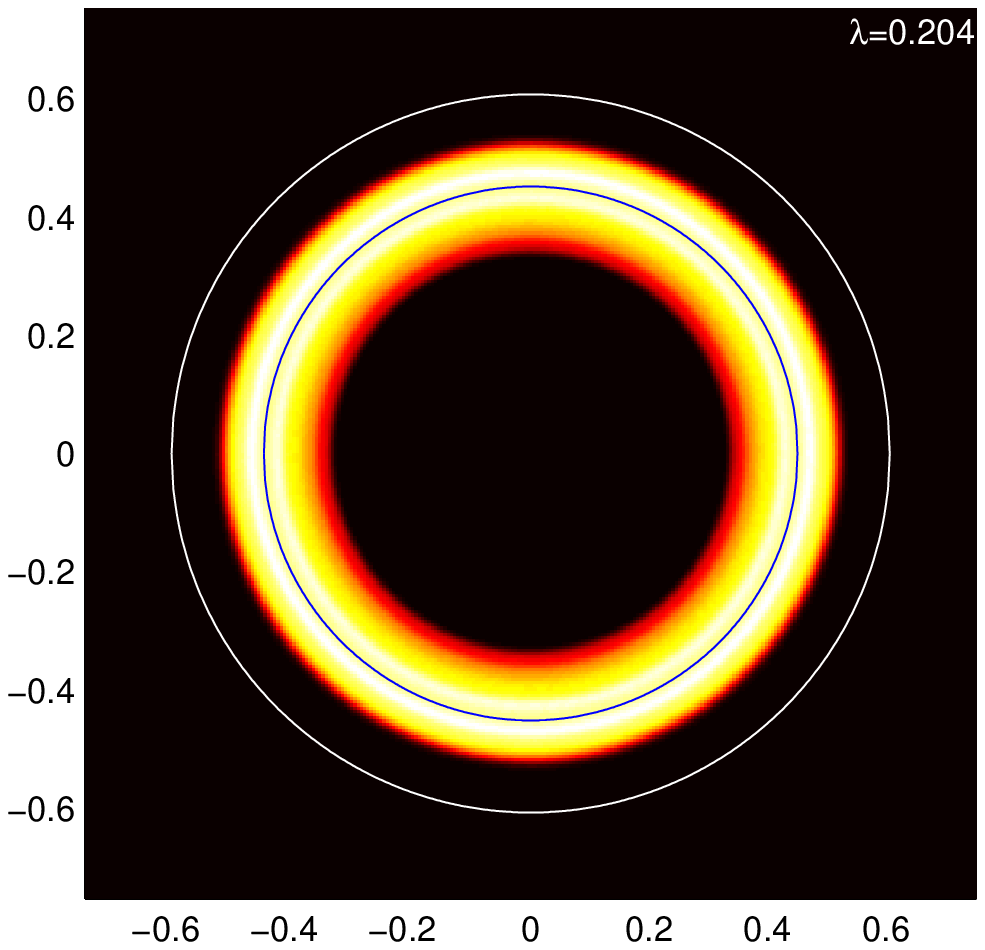,height=1.5in,width=2in}}}
\hspace*{-.3in}\hbox{{\psfig{figure=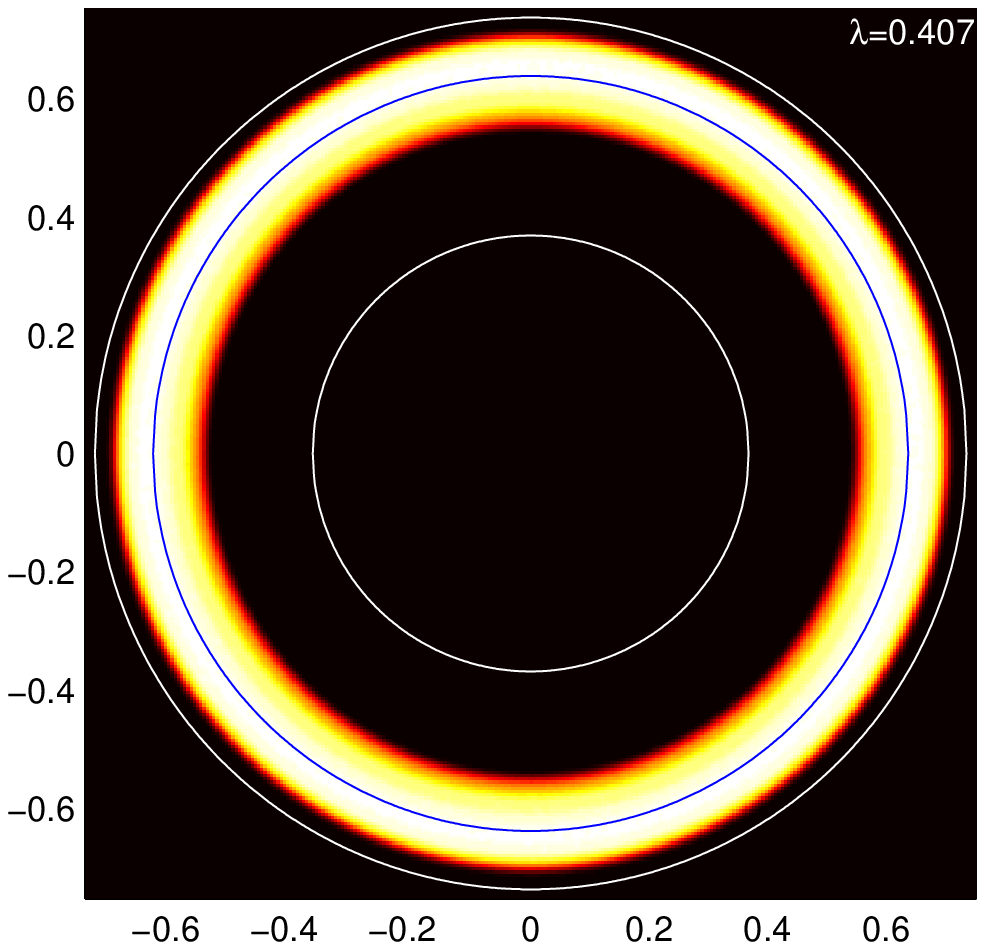,height=1.5in,width=2in}}}
\hspace*{-.3in}\hbox{{\psfig{figure=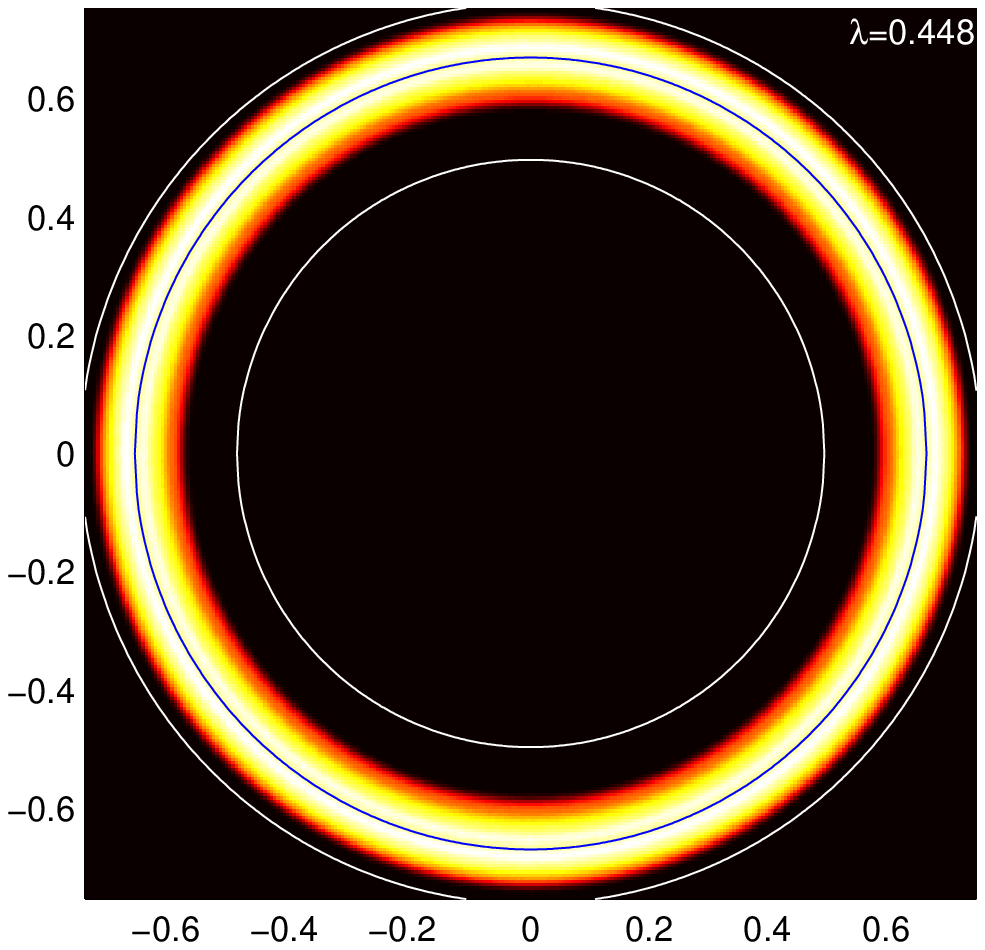,height=1.5in,width=2in}}}
}
\caption{Invariant densities for slow noise: $\sigma =.00001$ and $\varepsilon = 0.1$. 
}
\label{fig:slow}
\end{figure}

In the first set of simulations Figure~\ref{fig:fast}, $\sigma = 1.0$ as an example of fast noise and
in Figure~\ref{fig:slow}, $\sigma = 0.00001$ to show slow noise. Note that for slow noise 
and $\lambda >0$ there is a separation of time scales. Specifically, for each $\lambda$ and each 
set of values of $(u,v)$ the deterministic system defined by holding $(u,v)$ fixed has an
exponentially attracting periodic orbit. We expect then that the invariant density is concentrated
on the set which is the union of all of these deterministic limit cycles. For fast noise, this will 
not be the case and as is observed in Figure~\ref{fig:fast}, the approximated invariant densities
tend to be smoother  than for slow noise.

In both sets of simulations we use $\varepsilon = .1$. This leads to the following values of the 
bifurcation parameter and the radius of the inner boundary of the MFI set at the bifurcation:
$$
      \lambda_\mathrm{bif} = \frac{3 \varepsilon^{2/3}}{4^{1/3}} \approx 0.407163, \qquad
       r^* =  \left(\frac{\varepsilon}{2}\right)^{\frac{1}{3}} \approx 0.368403.
$$

We began the simulations by selecting random initial values for $(u,v)$ and $(x,y)$.
As we did not need an accurate solution to generate the noise we approximated
solutions to the SDE in (\ref{sde1})  using a first order
Taylor method, namely Euler's method.  We then added this noise into the
system in (\ref{sym_hopf_cart}), which we solved using a second order Adam's-Bashforth method.

After selecting the initial noise term and a starting point for the RDE,
we ran the system for 1000 iterations to allow the solution to the RDE to
move into the MFI sets.  We then ran the system for $5\times10^5$ iterations
and recorded the values of every fifth $(x,y)$ coordinate.  We repeated this
process for $10^2$ starting points and recorded a total of $10^7$ points, which
we used to generate a 2-dimensional histogram. Bright regions indicate that large numbers
of samples were observed in this region indicating a higher value of the invariant density. 
Darker regions indicate lower density.

The deterministic bifurcation takes place at $\lambda =0$ and the random bifurcation 
occurs at $\lambda_\mathrm{bif} \approx 0.4072$. In each set of figures, the densities are 
plotted for multiples of $\lambda_\mathrm{bif}$; 
namely for $\lambda = .01 \, \lambda_\mathrm{bif},  .05 \, \lambda_\mathrm{bif}, .1 \, \lambda_\mathrm{bif}, 
.5 \, \lambda_\mathrm{bif},  \lambda_\mathrm{bif}  $, and, $1.1 \, \lambda_\mathrm{bif}$. The outer circle
is the outer boundary of the MFI set, the circle inside the MFI is the stable periodic orbit 
of the deterministic system
and the inner circle that appears at $\lambda_\mathrm{bif}$ is the inner boundary of the MFI set.
It has initial radius $r^* = 0.3684$.

The invariant densities are positive for points inside the MFI set, but can be expected to
go to zero very rapidly as the boundary is approached (compare \cite{ZH}).
Also note that the densities become
undetectably small in the center long before the random bifurcation.

A color movie of the simulations can be found at:\\
http://www.youtube.com/watch?v=4tVtWGdVMi8.

\end{document}